\documentclass[11pt, reqno]{amsart}

\usepackage{version}
\usepackage{amsmath,amsfonts,amssymb,amsthm}
\usepackage{mathtools}
\usepackage{mathrsfs}
\usepackage{geometry}
\usepackage{nicefrac}
\usepackage[all]{xy}
\newtheorem{theorem}{Theorem}[section]
\newtheorem{lemma}[theorem]{Lemma}
\newtheorem{corollary}[theorem]{Corollary}
\newtheorem{proposition}[theorem]{Proposition}

\theoremstyle{definition}
\newtheorem{definition}[theorem]{Definition}
\newtheorem{example}[theorem]{Example}

\newtheorem{remark}[theorem]{Remark}
\newtheorem{question}[theorem]{Question}

\newcommand{\la}{\langle}
\newcommand{\ra}{\rangle}

\newcommand{\B}{\mathbb{B}}
\newcommand{\C}{\mathbb{C}}
\newcommand{\D}{\mathbb{D}}
\renewcommand{\H}{\mathbb{H}}
\newcommand{\N}{\mathbb{N}}

\newcommand{\R}{\mathbb{R}}

\newcommand{\Z}{\mathbb{Z}}

\def\de{\partial}

\renewcommand{\Re}{{\sf Re}\,}
\renewcommand{\Im}{{\sf Im}\,}

\numberwithin{equation}{section}

\title{The stable subset of a univalent self-map}
\author[L. Arosio]{Leandro Arosio$^\ast$}

\address{Dipartimento Di Matematica\\
Universit\`{a} di Roma \textquotedblleft Tor Vergata\textquotedblright\ \\
Via Della Ricerca Scientifica 1, 00133 \\
Roma, Italy} \email{arosio@mat.uniroma2.it}

\thanks{$^{*}$ Supported by the ERC grant ``HEVO - Holomorphic Evolution Equations'' n. 277691.}
\date{\today }


\begin{document}

\begin{abstract}
We give a complete description  of the stable subset  (the union of all backward orbit with bounded step) and of the pre-models of a univalent self-map $f\colon X\to X$, where $X$ is a Kobayashi hyperbolic cocompact complex manifold, such as the ball or the  polydisc in $\C^q$. 
The result is obtained studying the complex structure of a decreasing intersection of complex manifolds, all biholomorphic to $X$.
\end{abstract}
\maketitle
\tableofcontents

\section{Introduction}

Since the works of  Schr\"oder and K\"onigs \cite{S1,S2,K}, the idea of using  models in one complex variable to understand the behavior of the forward orbits of a holomorphic self-map has been proven to be very fruitful.  Recently, generalizing the work of Cowen in the unit disc, the concept of  model was introduced and studied in  general complex manifolds by Bracci and the author  \cite{AB}. We recall the main definitions and results.
Let $X$ be a complex manifold and let $f\colon X\to X$ be a holomorphic self-map. A {\sl  semi-model} for $f$ is given by a triple
$(\Omega, h,\psi)$, where $\Omega$ is a complex manifold, $h\colon X\to \Omega$ is a holomorphic mapping (called the {\sl intertwining mapping}), and $\psi\colon \Omega\to \Omega$ is an automorphism such that the following diagram commutes:
$$\xymatrix{X\ar[r]^{f}\ar[d]_{h}& X\ar[d]^{h}\\
\Omega\ar[r]^{\psi}& \Omega,}$$
and such that  $\bigcup_{n\geq 0} \psi^{-n}h(X)=\Omega$. 
A {\sl  model} for $f$ is a  semi-model such that the   intertwining  mapping $h\colon X\to \Omega$  is univalent on an $f$-absorbing domain $A\subset X$. 

Let $k_X$ denote the Kobayashi distance of $X$, and let $\kappa_X$ denote the Kobayashi metric of $X$. 
Notice that if $(z_n)$ is a forward orbit, then for all fixed $m\geq 1$  the sequence $(k_{X}(z_n,z_{n+m}))_{n\geq 0}$ is monotonically decreasing. The limit $$s_m(z_0)\coloneqq \lim_{n\to\infty} k_{X}(z_n,z_{n+m})$$ is called the {\sl forward $m$-step}. The {\sl divergence rate} of a self-map is a generalization introduced in  \cite{AB} of the dilation of a holomorphic self-map of the unit ball at the Denjoy--Wolff point. The following  result is proved in \cite{AB}.

\begin{theorem}[A.--Bracci]\label{pippo}
Let $X$ be Kobayashi hyperbolic and cocompact and let $f\colon X\to X$ be a univalent self-map. Then there exists an essentially unique  model $(\Omega,h,\psi)$. Moreover there  exists a holomorphic retract $Z$ of $X$, a  surjective holomorphic submersion $r\colon \Omega\to Z$, and an automorphism $\tau\colon Z\to Z$ with divergence rate   $$c(\tau)=c(f)=\lim_{m\to \infty}\frac{s_m(x)}{m},\quad x\in X,$$
such that $(Z,r\circ h,\tau)$ is a  semi-model for $f$.
Moreover $(Z,r\circ h,\tau)$ satisfies the following universal property. If $(Q, \ell, \varphi)$ is another  semi-model for $f$ such that $ Q$ is Kobayashi hyperbolic, then there exists a surjective holomorphic mapping $\eta\colon Z\to Q$ such that
 the following diagram commutes:

\SelectTips{xy}{12}
\[ \xymatrix{X \ar[rrr]^\ell\ar[rrd]^{r\circ h}\ar[dd]^f &&& Q \ar[dd]^\varphi\\
&& Z \ar[ru]^\eta \ar[dd]^(.25)\tau\\
X\ar'[rr]^\ell[rrr] \ar[rrd]^{r\circ h} &&& Q\\
&& Z \ar[ru]^\eta.}
\]
\end{theorem}

In this paper we study the dual concept: a {\sl pre-model} for $f$ is given by a triple
$(Q, t,\vartheta)$, where $Q$ is a complex manifold, $t\colon Q\to X$ is a holomorphic mapping (called the {\sl   intertwining mapping}) and $\vartheta\colon Q\to Q$ is an automorphism such that the following diagram commutes:
$$\xymatrix{Q\ar[r]^{\vartheta}\ar[d]_{t}& Q\ar[d]^{t}\\
X\ar[r]^{f}& X.}$$
A  pre-model for $f$ is {\sl injective} if the intertwining mapping $t\colon Q\to X$  is injective. 

Pre-models are naturally related to 
the {\sl backward orbits} of $f$, that is, the sequences  $(z_n)$ in $X$ such that $f(z_{n+1})=z_n$ for all $n\geq 0$.
Notice that if $(z_n)$ is a backward orbit, then  for all fixed $m\geq 1$ the sequence $(k_{X}(z_n,z_{n+m}))_{n\geq 0}$ is monotonically increasing. The limit $$\sigma_m(z_n)\coloneqq \lim_{n\to\infty} k_{X}(z_n,z_{n+m})$$   is called the {\sl backward $m$-step}. If
 $f$ is univalent, then $(z_n)$ is the unique backward orbit starting at $z_0$, and thus we simply write  $\sigma_m(z_0)$ instead of $\sigma_m(z_n)$.
A backward orbit $(z_n)$ has {\sl bounded step} if $\sigma_1(z_n)<\infty$.

In 1998 Poggi-Corradini  \cite{PC0}  studied injective pre-models for univalent self-maps of the  unit disc with a fixed point at the origin.
Backward orbits  were introduced in 2000  by Poggi-Corradini \cite{PC1} as a tool for constructing pre-models for holomorphic self-maps $f$ of the unit disc. The dynamics of backward orbits with bounded step was first studied in the unit disc in 2003  by Bracci and Poggi-Corradini \cite{B,PC2}. In 2005 Contreras and D\'iaz-Madrigal \cite{CD} obtained results similar to \cite{PC0} in the context of semigroups of the unit disc with no fixed point $z\in \D$.
Recently  Ostapyuk \cite{O} studied the case of the unit ball, and Abate--Raissy and Abate--Bracci studied the case of strictly convex domains \cite{AR, AbB}.

Given a boundary repelling fixed point $\zeta$ with dilation $1<\lambda<\infty$, Poggi-Corradini \cite{PC1} proved   that there exists a backward orbit $(z_n)$ converging to $\zeta$ satisfying $\sigma_1(z_n)=\log \lambda$. Using such a sequence he proved the following result (for the notions of  boundary repelling fixed points and dilations see Definition \ref{magamago'}).

 \begin{theorem}[Poggi-Corradini ]
Let $f\colon \D\to \D$ be  a holomorphic self-map, and let $\zeta\in \partial \D$ be a boundary repelling fixed point with dilation $1<\lambda<\infty$.
Then there exists a pre-model $(\H, t,\vartheta)$ for $f$ such that $\vartheta$ is the hyperbolic automorphism of $\H$ given by
$$\vartheta(z)=\frac{1}{\lambda} z,$$ 
and the mapping $t$ has non-tangential limit $\zeta$ at $\infty$.
\end{theorem}
Poggi-Corradini also proves that the pre-model $(\H, t,\vartheta)$ is essentially unique.
The same strategy was used by Ostapyuk to generalize this result to the unit ball $\B^q\subset \C^q$. She proved in \cite{O} that if a  boundary repelling fixed point $\zeta$ is isolated, then there exists  a backward orbit $(z_n)$ converging to $\zeta$ and satisfying $\sigma_1(z_n)=\log \lambda$. As a consequence, the following result is proved.
\begin{theorem}[Ostapyuk]\label{mila}
Let $f\colon \B^q\to \B^q$ a holomorphic self-map, and let $\zeta$ be a boundary repelling fixed point with dilation $1<\lambda<\infty$, which is isolated from other boundary repelling fixed points with dilation less or equal than $\lambda$. Then there exists a pre-model $(\H,t,\vartheta)$ for $f$ such that $\vartheta$ is the hyperbolic automorphism of $\H$ given by $$\vartheta(z)=\frac{1}{\lambda} z,$$
and the mapping $t$ has non-tangential limit $\zeta$ at $\infty$.
\end{theorem}
 Theorem \ref{mila} gives dynamical information on $f$ only on the one-dimensional image $t(\H)$. This remark motivated the following open question  \cite[Question 6.2.1]{O}. Recall that the {\sl stable subset} $\mathcal{S}(\zeta)$ at the boundary regular fixed point $\zeta$ is  the union of all backward orbits with bounded backward step that  tend to $\zeta$. 
\begin{question}\label{robinhoodintro}
Let $f\colon \B^q\to \B^q$ a holomorphic self-map, and let $\zeta$ be a boundary repelling fixed point with dilation $1<\lambda<\infty$.
In  one dimension, $\mathcal{S}(\zeta)=t(\H)$.
It is important to understand the properties of the stable subset at  boundary repelling fixed point in several variables, because it may help
to find the ``best possible'' intertwining map, i.e. the intertwining map whose image has the largest dimension.
\end{question}

We define the {\sl stable subset} $\mathcal{S}(f)$ of $f$ as  the union of all backward orbits with bounded  step in $X$.
Our main result  describes the  structure of the stable subset $\mathcal{S}(f)$ and the pre-models   for a univalent self-map $f$ of a Kobayashi hyperbolic cocompact  manifold. In particular we show that every backward orbit with bounded step gives rise to an essentially unique injective  pre-model, with the ``best possible'' intertwining map.
Our geometric approach  is completely new in this context and yields a result in  duality with Theorem \ref{pippo}.

\begin{theorem}\label{mainintro}
Let $X$ be Kobayashi hyperbolic and cocompact and let $f\colon X\to X$ be a univalent self-map.  Then the stable subset $\mathcal{S}(f)$, if non-empty, is
the disjoint union
 of completely invariant complex  submanifolds $$\Lambda=\bigsqcup_{j\in J} \Sigma_j,$$ such that 
 for all $j\in J$ there exists a holomorphic retract  $Z_j$ and an injective holomorphic immersion $g_j\colon Z_j\to X$ satisfying
 $g_j(Z_j)=\Sigma_j$.
For all $j\in J$, there exists an automorphism $\tau_j\colon Z_j\to Z_j$ with divergence rate $$c(\tau_j)=\lim_{m\to \infty}\frac{\sigma_m(x)}{m},\quad \forall\,x\in \Sigma_j,$$
such that $(Z_j,g_j,\tau_j)$ is an injective pre-model for $f$.
Moreover $(Z_j,g_j,\tau_j)$ satisfies the following universal property. If $(Q,t,\vartheta)$ is another pre-model for $f$ such that
$ t(Q)\cap\Sigma_j\neq \varnothing$, then there exists an injective holomorphic mapping $\eta\colon Z_j\to Q$ such that
 the following diagram commutes:
\SelectTips{xy}{12}
\[ \xymatrix{Z_j \ar[rrr]^{g_j}\ar[rd]^\eta\ar[dd]^{\tau_j} &&& X \ar[dd]^f\\
& Q \ar[rru]^t \ar[dd]^(.25)\vartheta\\
Z_j\ar'[r][rrr]^(.25){g_j} \ar[rd]^\eta &&& X\\
& Q \ar[rru]^t.}
\]
\end{theorem}

We apply this result to the case of the unit ball $\B^q$,
giving  the following answer to Question \ref{robinhoodintro}. Notice that if $X=\B^q$, then the holomorphic retract $Z_j$ is biholomorphic to a ball $\B^{k_j}$ with $0\leq k_j\leq q$. Recall that  the Siegel upper half-space $\H^q$ is biholomorphic to $\B^q$ (see Definition \ref{jafar}).
\begin{theorem}\label{timon}
Let $f\colon \B^q\to\B^q$ be a univalent self-map and let $\zeta\in \partial\B^q$ be a boundary repelling fixed point  with dilation $1<\lambda<\infty$.  
 Then the stable subset  $\mathcal{S}(\zeta)$ at $\zeta$, if non-empty, is
the disjoint union
 of completely invariant complex  submanifolds $$\Lambda=\bigsqcup_{j\in J} \Sigma_j.$$ 
Fix $j\in J$,  let $1\leq k_j\leq q$ be the dimension of $\Sigma_j$ and define $\mu_j$ by  $$\mu_j\coloneqq {\lim_{m\to \infty}e^\frac{\sigma_m(x)}{m}}\geq \lambda,$$ where  $x\in\Sigma_j$. 
Then $\mu_j$ does not depend on $x\in \Sigma_j$ and there exist
  an injective holomorphic immersion  $g_j\colon \H^{k_j}\to \B^q$ with $g_j(\H^{k_j})=\Sigma_j$ and $$K\hbox{-}\lim_{z\to \infty}g_j(z)=\zeta,$$
  and   a $(k_j-1)\times(k_j-1)$ diagonal unitary matrix $U_j$
 such that
 $$\left(\H^{k_j},g_j,\tau_j\colon(z,w)\mapsto \left(\frac{1}{\mu_j} z,\frac{1}{\sqrt {\mu_j}}\, U_jw\right)\right)$$ is an injective pre-model for $f$.
If $\Sigma_j$ contains a special and restricted backward orbit, then $\mu_j=\lambda$.
Moreover, the  pre-model $(\H^{k_j},g_j,\tau_j)$ satisfies the same universal property as in Theorem \ref{mainintro}.
\end{theorem}

The   proof of Theorem \ref{mainintro}  involves the study of the complex structure of a decreasing intersection. We recall some results for the dual problem, that is the study of the complex structure  of a growing union, also called the {\sl union problem}.
Assume that we have a monotonically increasing sequence of domains of a complex manifold $\Omega$:
$$X_0\subset X_1\subset X_2\subset\dots,$$ and assume that $\Omega=\bigcup_{n\geq0}X_n.$ Assume moreover that every $X_j$ is biholomorphic to a given complex manifold $X$. One wants to understand  the complex structure of $\Omega$. Forn\ae ss \cite{F} gave a surprising example with $X=\B^3$ which is not Stein. Later Forn\ae ss and Sibony  \cite{FS} gave a description of the complex structure of $\Omega$ which implies the following result.
\begin{theorem}[Forn\ae ss--Sibony]
If $X$ is Kobayashi hyperbolic and cocompact, then there exists a holomorphic retract $Z\subset X$ and a  surjective holomorphic submersion $r\colon \Omega\to Z$ which satisfies the following universal property. If $Q$ is a Kobayashi hyperbolic complex manifold and $t\colon \Omega\to Q$ is a holomorphic mapping, then there exists a holomorphic mapping $\sigma\colon Z\to Q$ such that 
the following diagram commutes:
$$\xymatrix{\Omega\ar[r]^{t}\ar[d]_{r}& Q\\
Z\ar[ru]_\sigma.& }$$
 Moreover   $$k_{X_n}\nearrow k_{\Omega}, \quad\mbox{and}\quad \kappa_{X_n}\nearrow \kappa_{\Omega}.$$
 \end{theorem}

Let now $X$ be a complex manifold, and assume there exists a monotonically decreasing sequence of domains $$X=X_0\supset X_1\supset X_2\supset\dots,$$ where every $X_j$ is biholomorphic to  $X$. 
Let $\Lambda$ denote the intersection $\bigcap_{n\geq0}X_n.$ What can be said about $\Lambda?$ Notice that $\Lambda$ can be empty and  in general is not a complex submanifold of $X$, so it is not  clear {\sl a priori} if it is endowed with a complex structure.
The next result  gives an answer to this problem, showing  that $\Lambda$ inherits some complex structure from the sequence $(X_j)$.

\begin{theorem}\label{carmen}
Let  $X$ be Kobayashi hyperbolic and cocompact. Then the subset $\Lambda\coloneqq \bigcap_{n\geq0}X_n$, if non-empty, is the disjoint union
 of complex  submanifolds $$\Lambda=\bigsqcup_{j\in J} \Sigma_j,$$ such that 
 for all $j\in J$ there exists a holomorphic retract  $Z_j$ and an injective holomorphic immersion $g_j\colon Z_j\to X$ satisfying
 $g_j(Z_j)=\Sigma_j$.
For all $j\in J$, the mapping $g_j$ satisfies the following universal property. If $Q$ is a complex manifold and $t\colon Q\to X$ is a holomorphic mapping such that $t(Q)\cap \Sigma_j\neq\varnothing $ and $t(Q)\subset \Lambda$, then
there exists a holomorphic mapping $\eta\colon Q\to Z_j$ such that the following diagram commutes:
$$\xymatrix{Q\ar[r]^{t}\ar[d]_{\eta}& X\\
Z_j\ar[ru]_{g_j}.& }$$
  Moreover, for all $j\in J$,  $$k_{X_n}|_{\Sigma_j}\searrow k_{\Sigma_j}, \quad\mbox{and}\quad \kappa_{X_n}|_{\Sigma_j}\searrow\kappa_{\Sigma_j}.$$
\end{theorem}

The plan of the paper is as follows.
 Section \ref{intersection} is devoted to the problem of decreasing intersections and to the proof of Theorem \ref{carmen}.
In Section \ref{simba} we introduce the  pre-models and the stable subset, and we prove Theorem \ref{mainintro}. Finally, in Section  \ref{mufasa} we consider the case of $\B^q$ and we prove  Theorem \ref{timon}.

\section{The complex structure of a decreasing intersection}\label{intersection}
In this section we prove Theorem  \ref{carmen}.
Throughout the paper, complex manifolds are assumed to be connected unless  otherwise specified.
\begin{definition}

Let $X,Z,Q$ be complex manifolds and let $g\colon Z\to X$ and $t\colon Q \to X$ be  holomorphic mappings.

We say that {\sl $g$ extends $t$} if there exists a holomorphic mapping   $\eta\colon Q\to Z $ such that the following diagram commutes:
$$\xymatrix{Q\ar[r]^{t}\ar[d]_{\eta}& X\\
Z\ar[ru]_g.& }$$

If the mapping $\eta\colon Q\to Z $ is a biholomorphism, we say that  $g$ and $t$ are {\sl equivalent}. Notice that in this case, $t$ extends $g$ through the mapping $\eta^{-1}$.
\end{definition}
\begin{remark}\label{genio}
Let $g\colon Z\to X$ and $f\colon Q \to X$ be injective holomorphic mappings. Assume that $g$ extends $t$ through the mapping  $\eta\colon Q\to Z $. 
Then   $\eta\colon Q\to Z $ is injective and unique.
\end{remark}

\begin{lemma}\label{mulan}
Let $X,Z,Q$ be complex manifolds and let $g\colon Z\to X$ and $f\colon Q \to X$ be injective holomorphic mappings.
Assume that $g$ extends $t$.
Then the following are equivalent:
\begin{enumerate}
\item $t$ extends $g$,
\item $g(Z)\subset t(Q)$,
\item $\eta\colon Q\to Z$ is surjective,
\item $t$ and $g$ are equivalent.
\end{enumerate}
\end{lemma}
\begin{proof}
The proof is trivial.
\end{proof}

Let $X$ be a Kobayashi hyperbolic complex manifold. We say that $X$ is {\sl cocompact} if $X/{\rm aut}(X)$ is compact. Notice that this implies that $X$ is complete hyperbolic \cite[Lemma 2.1]{FS}.   Assume that  there exists a monotonically decreasing sequence of domains $$X=X_0\supset X_1\supset X_2\supset\dots,$$ where every $X_j$ is biholomorphic to  $X$, and  let $f_j\colon X\to X_j$ be a biholomorphism. Let $\Lambda$ denote the intersection $\bigcap_{n\geq0}X_n.$ 

\begin{remark}
Let $k_X$ denote the Kobayashi distance of $X$, and let $\kappa_X$ denote the Kobayashi metric. Let $x,y\in \Lambda$ and let $v\in T_xX$.
Then the sequences $(k_{X_n}(x,y))_{n\geq 0}$ and $(\kappa_{X_n}(x,v))_{n\geq 0}$ are monotonically increasing.
\end{remark}

\begin{remark}\label{immer}
Let $Q$ be a complex manifold and let $t\colon Q\to X$ be a holomorphic mapping with  $t(Q)\subset \Lambda$. Then $f_n^{-1}\circ t\colon Q\to X$ is a well-defined holomorphic mapping for all $n\geq 0$.
\end{remark}
\begin{definition}\label{catwoman}
We define an equivalence relation $\sim$ on $\Lambda$ in the following way: $x,y\in \Lambda$ are equivalent if and only if the sequence $(k_{X_n}(x,y))_{n\geq 0}$ is  bounded. Notice that $k_{X_n}(x,y)=k_{X}(f_n^{-1}(x),f_n^{-1}(y))$. The class of $x$ will be denoted by $[x]$. 
\end{definition}
\begin{remark}\label{pizza}
Definition \ref{catwoman} is independent of the biholomorphisms $(f_j\colon X\to X_j)$ chosen. 

\end{remark}

\begin{lemma}\label{giallo}
Let $Q$ be a complex manifold and let  $t\colon Q\to X$ be a holomorphic mapping with  $t(Q)\subset \Lambda$. Let $x,y\in t(Q)$. Then the sequence $(k_{X}(f_n^{-1}(x),f_n^{-1}(y)))$ is bounded, that is $x\sim y$.
Similarly, if $x= t(z)$ and $v= d_zt(\zeta)$ with $\zeta\in T_zQ$, then the sequence $(\kappa_{X}(f_n^{-1}(x),d_xf_n^{-1}(v)))$ is bounded.
\end{lemma}
\begin{proof}
Let  $x,y\in t(Q)$ and let $z,w\in Q$ such that $t(z)=x$ and $t(w)=y$. By Remark \ref{immer} we have, for all $n\geq0,$ $$k_{X}(f_n^{-1}(x),f_n^{-1}(y))\leq k_Q(z,w).$$
Similarly, again by Remark \ref{immer}, for all $n\geq 0$, $$\kappa_{X}(f_n^{-1}(x),d_xf_n^{-1}(v))\leq\kappa_{Q}(z,\zeta).$$
\end{proof}

Let $ x_0\in \Lambda$. Since $X$ is cocompact, there exist a compact subset $K\subset X$ and a family of automorphisms $(\tau_n\colon X\to X)$ such that $(f_n\circ \tau_n)^{-1}(x_0)\in K$ for all $n\geq 0$.
For all $n\geq 0$, denote $g_n\coloneqq f_n\circ \tau_n$.
Since by construction $x_0\in g_n(K)$ for all $n\geq 0$,  the sequence $(g_n\colon X\to X)$ does not diverge on compact subsets. Since $X$ is complete hyperbolic, it is taut, and hence up to taking a subsequence  $(g_n)$ converges uniformly on compact subsets to a holomorphic mapping $g\colon X\to X$.

\begin{lemma}\label{tre}
We have that $g(X)=[x_0]$.
\end{lemma}
\begin{proof}
We  show that $[x_0]\subset g(X)$.
Let $y\in [x_0]$. Then the sequence $(k_X(g_{n}^{-1}(x_0),g_{n}^{-1}(y)))$ is bounded by $M>0$.
Since $X$ is complete hyperbolic, the subset $\{x\in X \colon k_X(x,K)\leq M\}$ is compact. This implies that up to a subsequence, $z_n\coloneqq g_{n}^{-1}(y)\to z\in X$, and
$$y=g_{n}(z_n)\to g(z).$$
We now show that  $g(X)\subset \Lambda$, and then Lemma \ref{giallo} yields $g(X)\subset[x_0]$. Let thus $n\geq 0$.  
Consider the sequence of holomorphic mappings $(g_m\colon X\to g_n(X))_{m\geq n}$. Since $x_0\in g_m(K)$ for all $m\geq n$, this sequence is not compactly divergent.  Since $g_n(X)$  is biholomorphic to $X$, it is taut, and hence up to taking a subsequence  $(g_m\colon X\to g_n(X))_{m\geq n}$ converges uniformly on compact subsets to a holomorphic mapping $\tilde g\colon X\to g_n(X)$. Since $g=\tilde g$, we have the result.  

\end{proof}

Consider now the sequence of holomorphic mappings $(\alpha_n\colon X\to X)$ defined by $$\alpha_n\coloneqq  g_n^{-1}\circ g.$$ 
Let $x'\in g^{-1}(x_0)$.
For all $n\geq 0$ we have $$\alpha_n(x')=g_n^{-1}(x_0)\in K,$$ hence $(\alpha_n)$ does not diverge on compact subsets, and thus  we can assume (up to taking a subsequence) that $\alpha_n$ converges uniformly on compact subsets to a holomorphic mapping $\alpha\colon X\to X$.
\begin{remark}\label{shiva}
If $z= g(x)$, then the sequence $g^{-1}_n(z)$ converges to the point $\alpha(x)$ in $X$.
If $\zeta\in T_xX$ and $v=d_xg(\zeta)\in T_zX$, then $d_zg^{-1}_n(v)\to d_x\alpha(\zeta).$
\end{remark}

\begin{lemma}\label{yob}
The map $\alpha\colon X\to X$ is a holomorphic retraction.
\end{lemma}
\begin{proof}
Let $x\in X$ and let $z=g(x)$.
Let $z_n\coloneqq g_n^{-1}(z)$. By Remark \ref{shiva}, $$g(x)=z=g_n(z_n)\to g(\alpha (x)).$$
We have that $\alpha_n( \alpha(x))\to \alpha(\alpha(x))$. Again by Remark \ref{shiva}, $$\alpha_n( \alpha(x))=g_n^{-1}(g(\alpha(x)))=g_n^{-1}(g(x))\to\alpha(x).$$
\end{proof}

\begin{remark}\label{merlino}
Denote $Z=\alpha(X)$. Since $\alpha$ is a holomorphic retraction, by \cite[Lemma 2.1.28]{A} the image  $Z$ is a closed complex submanifold of $X$.
In particular, if $X$ is the unit ball $\B^q$, then  by \cite[Corollary 2.2.16]{A} the retract $Z$ is an affine subset and hence biholomorphic to ball $\B^k$ with $0\leq k\leq q$. If $X$ is the polydisc $\Delta^q$, then by \cite[Theorem 3]{HS} the retract $Z$ is biholomorphic to a  polydisc $\Delta^k$ with $0\leq k\leq q$.
\end{remark}

\begin{lemma}
We have that $g(X)=g(Z)$ and that $g|_Z\colon Z\to X$ is an injective holomorphic immersion.

\end{lemma}
\begin{proof}

If $x\in X$, then $g(x)=g(\alpha(x))$ and $\alpha(x)\in Z$, thus $g(X)=g(Z)$.
Assume that $x,y\in Z$ and that $g(x)=g(y)$. Then $\alpha_n(x)=\alpha_n(y)$ for all $n\geq 0$, and thus $\alpha(x)=\alpha(y)$, which implies $x=y$. Hence  $g|_Z\colon Z\to X$ is injective.

We are left to prove that $g|_Z\colon Z\to X$ is an immersion.
Let  $x\in X$. Since $\alpha$ is the uniform limit on compact subsets of the sequence $g_n^{-1}\circ g$, we have by the lower semicontinuity of the rank that ${\rm rk}_x(g)\geq  {\rm rk}_x(\alpha)$.
 On the other hand we have that $${\rm rk}_x(g)={\rm rk}_x(g|_Z\circ \alpha)\leq {\rm min}\{{\rm rk}_{\alpha(x)}(g|_Z),{\rm rk}_x(\alpha)\}.$$ Then for all $z\in Z$ we have ${\rm rk}_z(g)={\rm rk}_z(g|_Z)= {\rm rk}_z(\alpha)={\rm dim}_z Z$.

\end{proof}

In what follows we set $\Sigma\coloneqq g(Z)$.
\begin{proposition}\label{mrpink}

Let $Q$ be a complex manifold and let $t\colon Q \to X $ be a holomorphic mapping such that $t(Q)\cap\Sigma \neq \varnothing$ and $t(Q)\in\Lambda$. Then $g\colon Z\to X$ extends $t$.

\end{proposition}

\begin{proof}
 By Remark \ref{shiva} the sequence $(g_n^{-1}\circ t\colon Q\to X)$ converges pointwise to a map   $\eta\colon  Q\to X$ such that  $g\circ \eta=t$ and $\eta(Q)\subset Z$. Since $X$ is taut, by Vitali's Theorem the convergence is uniform on compact subsets and  $\eta$ is holomorphic. 
\end{proof}

\begin{corollary}\label{aristogatti}

Let $Q$ be a complex manifold and let $t\colon Q\to X$ be an injective holomorphic mapping such that $t(Q)=\Sigma$. Then $t\colon Q\to X$ and $g\colon Z\to X$ are equivalent.
\end{corollary}
\begin{proof}
By Proposition \ref{mrpink} we have that $g$ extends $t$.  Lemma \ref{mulan} yields the result.
\end{proof}

\begin{definition}
We call the injectively immersed complex submanifold $\Sigma\subset X$ a {\sl canonical submanifold}.
\end{definition}

\begin{example}
Let
 $$\D=X_0\supset X_1\supset X_2\supset\dots,$$ be  a monotonically decreasing sequence of simply connected domains in the unit disc. Then $\Lambda\coloneqq \bigcap_{n\geq 0}X_n$, if non-empty, is the disjoint union of the following canonical submanifolds:
each connected component of the interior part $\mathring{\Lambda}$ is simply connected and is a canonical submanifold. If $x\in \Lambda\smallsetminus \mathring{\Lambda}$, then $\{x\}$ is a canonical submanifold.
\end{example}

\begin{definition}
The {\sl tangent space}  of the canonical submanifold   $\Sigma$  at a point $x=g(z)$ is defined as $$T_x\Sigma\coloneqq dg_z(T_zZ)\subset T_x X.$$
\end{definition} It is natural to ask whether this complex subspace coincides with the subset $$V_x\coloneqq \{v\in T_xX\colon (\kappa_{X}(g_{n}^{-1}(x),d_xg_n^{-1}(v)))\ \mbox{is bounded}\}.$$ The following result shows that this is indeed the case.
Notice that {\sl a priori} it is not even  clear that $V_x$ is a complex vector subspace of $T_xX$: from the properties of the Kobayashi metric we can only infer that it is a complex cone. 

\begin{proposition}
For all $x\in \Lambda$ we have that $V_x= T_x \Sigma$.
\end{proposition}
\begin{proof}
From Lemma \ref{giallo} it follows that $T_x \Sigma\subset V_x$. 
For the converse inclusion, let $v\in V_x$. Assume that the sequence $ (\kappa_{X}(g_{n}^{-1}(x),d_xg_n^{-1}(v)))$ is bounded by $N>0$ and let $M\coloneqq N+1$. 
Then by definition of the Kobayashi metric (see, e.g., \cite[(3.5.16)]{Kob}), there exists a family of holomorphic mappings $(r_j)_{j\geq 0}\colon \D\to X$ such that $r_j(0)=g_j^{-1}(x)$ for all $j\geq 0$ and $r_j'(0)=\frac{2}{M}d_xg_j^{-1}(v)$. Define a family of holomorphic mappings from the disc $\D$ to $X$ setting $t_j=g_j\circ r_j$ for all $j\geq 0$. Then $t_j(0)=x$ and $t'_j(0)=\frac{2}{M}v$ for all $j\geq 0$. Since $X$ is taut, there exists a holomorphic mapping $t\colon \D\to X$ with $t(0)=x$ and $t'(0)=\frac{2}{M}v$ such that $t_j\to t$ uniformly on compact subsets. For all $n\geq 0$ the image $t_j(\D)$ is eventually contained in $g_n(X)$. Since all the domains $g_n(X)$ are taut, arguing as in Lemma \ref{tre} we obtain that the image $t(\D)$ is contained in $\Lambda$. By Proposition \ref{shiva} there exists a holomorphic mapping $\ell\colon \D\to Z$ such that 
 \begin{equation}\label{oi}
 t=g\circ \ell.
 \end{equation}
  Moreover, $\frac{2}{M}v=t'(0)=g'(\ell'(0))$, hence $\frac{2}{M}v$ belongs to $ T_x \Sigma$. Since $ T_x \Sigma$ is a complex subspace, the result follows.
 \end{proof}

\begin{remark}
Since $Z$ is a holomorphic retract of $X$, for all $x,y\in Z$ and $\zeta\in T_xZ$ we have  $k_Z(x,y)=k_X(x,y)$ and $\kappa_Z(x,\zeta)=\kappa_X(x,\zeta)$. Moreover $Z$ is complete  hyperbolic.
\end{remark}

\begin{definition}
Let $\Sigma$ be a canonical submanifold. Then its Kobayashi distance $k_\Sigma$ is defined as
$$\forall x,y\in\Sigma,\quad k_{\Sigma}(x,y)\coloneqq k_Z(z,w),$$ where $g(z)=x$, $g(w)=y$, and its Kobayashi metric $\kappa_\Sigma$ is defined as
$$\forall x \in\Sigma,v\in T_x\Sigma,\quad \kappa_{\Sigma}(x,v)\coloneqq \kappa_Z(z, \zeta),$$ where $g(z)=x$ and $d_z g (\zeta)=v.$
\end{definition}
Next we show that  the Kobayashi distance  on canonical submanifolds is approximated by $k_{X_j}$, and that the same holds for  the Kobayashi metric.

\begin{proposition}\label{mryellow}
Let $z,w\in \Sigma$ and let $v\in T_z\Sigma$. Then $$k_{X_n}(z,w)\to k_{\Sigma}(z,w), \quad\mbox{and}\quad \kappa_{X_n}(z,v)\to\kappa_{\Sigma}(z,v).$$

\end{proposition}
\begin{proof}
 Let  $x,y\in Z$ be such that $g(x)=z$ and $g(y)=w$.
By Remark \ref{shiva}, $$k_{X_n}(z,w)=k_X(g_n^{-1}(z), g_n^{-1}(w))\to k_X(\alpha(x),\alpha(y))=k_X(x,y).$$  Since $Z$ is a holomorphic retract of $X$, we have $k_X(x,y)=k_Z(x,y)\coloneqq k_{g(Z)}(z,w)$. 

Similarly, let  $\zeta\in T_xZ$ be such that $dg_x(\zeta)=v$. By Remark \ref{shiva},
$$\kappa_{X_n}(z,v)=\kappa_X(g_n^{-1}(z), d_zg_n^{-1}(v))\to \kappa_X(\alpha(x), d_x\alpha(\zeta))=\kappa_X(x,\zeta).$$
Since $Z$ is a holomorphic retract of $X$, we have $ \kappa_X(x,\zeta)=\kappa_Z(x,\zeta)\coloneqq \kappa_{g(Z)}(z,v)$. 

\end{proof}

The proof of Theorem \ref{carmen} follows easily from the results of this section.

\section{Pre-models and the stable subset}\label{simba}

In this section we prove Theorem  \ref{mainintro}.
We begin  with few basic properties of the backward orbits for self-maps of sets.

\begin{definition}
Let $X$ be a set and let $f\colon X\to X$ be a self-map. 
A {\sl backward orbit} is a sequence $(x_n)$ in $X$ such that $f(x_{n+1})=x_n$ for all $n\geq 0$.
The point $x_0$ is called the {\sl starting point} of the backward orbit $(x_n)$.
We denote $\Lambda\coloneqq \bigcap_{n\geq 0}f^n(X)$, and
we denote ${\rm BO}(f)$ the subset consisting of all $x\in X$ such that there exists a backward orbit starting at $x$. 
\end{definition}
Clearly ${\rm BO}(f)$ coincides with the union of all  backward orbits in $X$.

\begin{remark}\label{paperoga}
We have that ${\rm BO}(f)\subset \Lambda$.
\end{remark}

\begin{definition}
If $f\colon X\to X$ is injective and $y\in f(X)$, then we denote by $f^{-1}(y)$ the only point $x\in X$ such that $f(x)=y$.
\end{definition}
\begin{lemma}\label{basic}
If $f$ is injective, then ${\rm BO}(f)= \Lambda$.
\end{lemma}
\begin{proof}
We just need to prove that every  $x\in \Lambda$ is the starting point of a backward orbit. We claim that $(f^{-n}(x))$ is a backward orbit starting at $x$. Indeed, for all $n\geq 0$, from $f^{n+1}(f^{-n-1}(x)))=f^n(f^{-n}(x))$ we get $ f(f^{-n-1}(x)))=f^{-n}(x)$, as claimed.
\end{proof}

\begin{definition}
Let $X$ be a set and let $f\colon X\to X$ be a self-map. A subset $Y\subset X$ is {\sl forward invariant} if $f(Y)\subset Y$.  A subset $Y\subset X$ is {\sl completely invariant} if $f(Y)\cup f^{-1}(Y)\subset Y$.
\end{definition}
Clearly, both $\Lambda$ and ${\rm BO}(f)$ are  forward invariant. If $f$ is injective, we can say more.
\begin{lemma}\label{gaston}
 If $f$ is injective, then $\Lambda$ is completely invariant, and  the mapping $f|_\Lambda\colon\Lambda \to\Lambda$ is bijective.
\end{lemma}
\begin{proof}
The first statement is trivial. Let $f$ be injective, and let $x\in \Lambda$. 
By Remark \ref{paperoga} we have $f^{-1}(x)\in \Lambda$, which implies that $\Lambda$ is completely invariant and that the mapping $f|_\Lambda\colon\Lambda \to\Lambda$ is surjective. Since the mapping $f|_\Lambda$ is clearly injective, we have the result.

\end{proof}
\begin{definition}
If $f$ is injective, we call $\Lambda$ the {\sl invariant subset} of $f$.

\end{definition}

The following example is a modification of \cite[Example 2.1]{G} and shows that if $f$ is not injective, then Lemma  \ref{basic} and  Lemma \ref{gaston} do no longer  hold. 
\begin{example}
Let $A\subset \Z^2$ be the subset defined by $\{(n,m)\in \Z^2\colon 0\leq m\leq n\}$. Let $X$ be the disjoint union of $A$ with two singletons $\{p\}$ and $\{q\}$. We define a self-map $f\colon X\to X$ in the following way. On $ A\smallsetminus \{m=0\}$ set $f(n,m)\coloneqq (n,m-1),$ and on $A\cap\{m=0\}$ set $f(n,0)\coloneqq p$. Finally  set $f(p)\coloneqq q$ and $f(q)\coloneqq q$.
Then ${\rm BO}(f)=\{q\}$, while $\Lambda=\{p\}\cup\{q\}$. Moreover, $\Lambda$ is not completely invariant, and $f|_\Lambda\colon\Lambda \to\Lambda$ is neither injective nor surjective. Notice that  ${\rm BO}(f)$ is not completely invariant either.
\end{example}
See also \cite[Proposition 4.1]{G} for a  holomorphic self-map $f\colon \D\to \D$ such that ${\rm BO}(f)\neq \Lambda$.
 It is also easy to construct an example where $f\colon {\rm BO}(f)\to {\rm BO}(f)$ is not injective (clearly  it is  always surjective).

We now move back to  holomorphic self-maps of complex manifolds in order to introduce the pre-models and the morphisms between them.
\begin{definition}
Let $X$ be a complex manifold and let $f\colon X\to X$  be a holomorphic self-map. A {\sl pre-model} is a triple $(Q,t,\vartheta)$ such that $Q$ is a complex manifold, $t\colon Q\to X$ is a holomorphic mapping and $\vartheta\colon Q\to Q$ is an automorphism such that the following diagram commutes:

\[\xymatrix{Q\ar[r]^{\vartheta}\ar[d]_{t}& Q\ar[d]^{t}\\
X\ar[r]^{f}& X.}\]
A pre-model $(Q,t,\vartheta)$ is called   {\sl injective} if $t\colon Q\to X$ is injective.

Let $(Z,g,\tau)$ and  $(Q,t,\vartheta)$ be two pre-models for $f$. A  {\sl morphism of pre-models} $\hat\eta\colon  (Q,t,\vartheta) \to (Z,g,\tau) $ is given by a holomorphic mapping  $\eta \colon Q\to Z$ such that the following diagram commutes:
\SelectTips{xy}{12}
\[ \xymatrix{Q \ar[rrr]^t\ar[rd]^\eta\ar[dd]^\vartheta &&& X \ar[dd]^f\\
& Z \ar[rru]^g \ar[dd]^(.25)\tau\\
Q\ar'[r][rrr]^(.25)t \ar[rd]^\eta &&& X\\
& Z \ar[rru]^g.}
\]
If the mapping $\eta \colon Q\to Z$ is a biholomorphism, then we say that $\hat\eta\colon  (Q,t,\vartheta) \to (Z,g,\tau) $ is an {\sl isomorphism of pre-models}. Notice that then $\eta^{-1}\colon Z\to Q$ induces a morphism $ {\hat\eta}^{-1}\colon  (Z,g,\tau)\to  (Q,t,\vartheta). $

\end{definition}
\begin{lemma}\label{belle}
Let $(Z,g,\tau)$ be an injective  pre-model for $f$ and  let $(Q,t,\vartheta)$ be a  pre-model for $f$. Then
 there exists a morphism   $\hat\eta\colon  (Q,t,\vartheta) \to (Z,g,\tau)  $ if and only if  $g\colon Z\to X$ extends $t\colon Q\to X$ through the mapping $\eta\colon Q\to Z$.
\end{lemma}
\begin{proof}
One direction is trivial. For the other, assume that $g\colon Z\to X$ extends $t\colon Q\to X$ through the mapping $\eta\colon Q\to Z$.
Then $$g\circ \tau\circ \eta=f\circ g\circ \eta=f\circ t=t\circ \vartheta=g\circ \eta\circ \vartheta.$$ Since $g$ is injective, it follows that $ \tau\circ \eta=\eta\circ \vartheta.$

\end{proof}
\begin{corollary}
Let $(Z,g,\tau)$ and $(Q,t,\vartheta)$ be  two injective pre-models for $f$. Then $ (Q,t,\vartheta)$ and $ (Z,g,\tau)  $ are isomorphic if and only if  $g$ and $t$ are equivalent. 
\end{corollary}
\begin{proof}
 One direction is trivial. For the other, assume that $g$ and $t$ are equivalent.  Thus $g\colon Z\to X$ extends $t\colon Q\to X$ through the biholmorphism $\eta\colon Q\to Z$.   By Lemma \ref{belle}, it induces a morphism   $\hat\eta\colon  (Q,t,\vartheta) \to (Z,g,\tau)  $ which is thus an isomorphism.
\end{proof}
\begin{remark}
If $(Z,g,\tau)$ is an injective pre-model for $f$, then $Z$ is Kobayashi hyperbolic.

\end{remark}

\begin{definition}
Let $X$ be a complex manifold and let $f\colon X\to X$  be a holomorphic self-map.
Let $(x_n)$ be a backward orbit for $f$.  For  $m\geq 1$ we denote by
\begin{equation}\label{pigna}
\sigma_m(x_n)\coloneqq  \lim_{n\to\infty}k_X(f^{-n-m}(x),f^{-n}(x))\in \R^+\cup\{+\infty\}
\end{equation}
and we call it the {\sl backward $m$-step} of $(x_n)$. 
A backward orbit $(x_n)$ has {\sl bounded step} if $\sigma_1(x_n)<\infty$. If $f$ is univalent,
 then $(x_n)$ is the unique backward orbit starting at $x_0$, and thus we simply write  $\sigma_m(x_0)$ instead of $\sigma_m(x_n)$.
\end{definition}
Notice that the limit in (\ref{pigna}) exists since the sequence $$(k_X(f^{-n-m}(x),f^{-n}(x)))_{n\geq 0}$$ is monotonically increasing.

\begin{definition}
Let $f\colon X\to X$ be a holomorphic self-map. We define the {\sl stable subset} $\mathcal S(f)$
as the subset consisting of all $x\in X$ such that there exists a backward orbit with bounded step starting at $x$. 
\end{definition}
Clearly $\mathcal S(f)$ coincides with the union of all backward orbits in $X$ with bounded  step.

\begin{remark}
Let $f\colon X\to X$ be a holomorphic self-map and let $(Q,t,\vartheta)$ be a pre-model for $f$. Then
the image $t(Q)$  is  contained in the stable subset $\mathcal{S}(f)$.
\end{remark}

From now on  assume that $X$ is a Kobayashi hyperbolic cocompact complex manifold, and that $f\colon X\to X$ is a  univalent self-map. The invariant subset $\Lambda$ of $f$ is the decreasing intersection of the domains $(f^n(X))$. Each domain is biholomorphic to $X$ since 
$f^n$ is univalent.   We can thus apply the results of Section \ref{intersection} to the case $X_n\coloneqq f^n(X)$ and $f_n\coloneqq f^n$. On $\Lambda$ we consider the   equivalence relation $\sim$ introduced in Definition \ref{catwoman}:
$x,y\in \Lambda$ are equivalent if and only if  $(k_{X}(f^{-n}(x),f^{-n}(y)))$ is a bounded sequence.

\begin{remark}\label{principegiovanni}
The equivalence relation $\sim$ on the invariant subset $\Lambda$ is preserved by $f$ in the following sense: $x\sim y$ if and only if $f(x)\sim f(y)$.
\end{remark}

This means that $f$ induces a bijection (still denoted by $f$) on the family of canonical submanifolds $\Lambda/\sim$.
\begin{definition}
A canonical submanifold $\Sigma$ is {\sl invariant} if $f(\Sigma)=\Sigma$.
If $q\geq 2$, a canonical submanifold $\Sigma$ is {\sl $q$-periodic} if $f^k(\Sigma)=\Sigma$ for $k\in q\Z$.
A canonical submanifold $\Sigma$ is {\sl wandering} if $f^k(\Sigma)\neq \Sigma$ for all $k\in \Z$, $k\neq 0$.
\end{definition}

\begin{example}
We construct a  self-map of a domain biholomorphic to the unit disc $\D$  which admits one-dimensional  wandering canonical submanifolds. Consider the domain $A\subset \C$ defined by 
$$A\coloneqq\{0<\Im z<e^{\Re z}\}\cup (\C\smallsetminus\{\Re z \in \Z\}).$$ The domain $A$ is simply connected and thus by the uniformization theorem it is biholomorphic to the unit disc. Moreover it is invariant by the mapping $f(z)=z+1$, which thus defines a univalent mapping on $A$. We have $\Lambda= \{\Im z=0\}\cup  (\C\smallsetminus\{\Re z \in \Z\}),$ and for all $n\in \Z$ the  vertical stripe $A_n\coloneqq\{n<\Re z<n+1\}$ is a canonical submanifold. Clearly each $A_n$ is  wandering.
\end{example}

We can detect whether a canonical submanifold  $\Sigma$ is invariant, periodic or wandering just by looking at the backward steps $(\sigma_m(x))$ of  any point $x\in \Sigma$.
\begin{lemma}\label{slayer}
Let $\Sigma$ be a canonical submanifold and let $x\in\Sigma$. Then
\begin{enumerate}
\item  $\Sigma$ is invariant if and only if $\sigma_1(x)<\infty$, 
\item $\Sigma$ is $q$-periodic if and only if $\sigma_k(x)<\infty$ for all $k\in q\Z$,
\item $\Sigma$ is wandering if and only if $\sigma_k(x)=\infty$ for all $k\in \Z$, $k\neq 0$.

\end{enumerate}
\end{lemma}
\begin{proof}
If $x\in \Lambda$, we have that $\sigma_k(x)<\infty$ if and only if $x\sim f^{-k}(x)$ if and only if $x\sim f^{k}(x)$.
\end{proof}

\begin{remark}
The stable subset $\mathcal S(f)$ is the disjoint union of all invariant canonical submanifolds,
and the image $t(Q)$ of the intertwining mapping of a pre-model $(Q,t,\vartheta)$  is contained in an invariant canonical submanifold.

\end{remark}

Let $\Sigma$ be an invariant canonical submanifold. Then by Lemma \ref{gaston} and  Remark \ref{principegiovanni}, we have that $f|_\Sigma\colon \Sigma\to\Sigma$ is bijective. If $\Sigma$ is embedded, this implies that $f|_\Sigma\colon \Sigma\to\Sigma$ is an automorphism. Surprisingly enough, even if $\Sigma$ is not embedded we  obtain a similar result using   Proposition \ref{mrpink}.

\begin{proposition}\label{dumbo}
Let $\Sigma$ be an invariant canonical submanifold, and let $g\colon Z\to X$ be a holomorphic injective mapping such that $g(Z)=\Sigma$. 
Then there exists an automorphism $\tau$ of $Z$ such that the following diagram commutes:
$$\xymatrix{Z\ar[r]^{\tau}\ar[d]_{g}& Z\ar[d]^{g}\\
X\ar[r]^{f}& X.}$$
The  injective pre-model $(Z,g,\tau)$ satisfies the following properties:
\begin{enumerate} 
\item if $(Q,t,\vartheta)$ is another pre-model such that $ t(Q)\cap \Sigma\neq \varnothing$, then there exists a morphism  $\hat\eta\colon (Q,t,\vartheta) \to  (Z,g,\tau)$;
\item  if $(Q,t,\vartheta)$ is another injective pre-model such that $ t(Q)=\Sigma$, then there exists an isomorphism  $\hat\eta\colon (Q,t,\vartheta) \to  (Z,g,\tau)$.

\end{enumerate}

\end{proposition}
\begin{proof}
Consider the holomorphic mapping $f\circ g\colon Z\to X$. Since $\Sigma$ is invariant, we have $f(g(Z))\subset \Sigma$. By Proposition \ref{mrpink}, there exists a holomorphic mapping $\tau\colon Z\to Z$ such that $$ f\circ g  =g\circ \tau.$$ The mapping $\tau$ is bijective and is hence an automorphism.  Thus  $(Z,g,\tau)$  is an injective pre-model for $f$.
Let $(Q,t,\vartheta)$ be another pre-model such that $ t(Q)\cap \Sigma\neq \varnothing$.
By Proposition \ref{mrpink} we have that $g$ extends $t$, and by Lemma \ref{belle} 
there exists   a morphism of  pre-models  $\hat\eta\colon (Q,t,\vartheta) \to  (Z,g,\tau)$. 
Let $(Q,t,\vartheta)$ be another injective pre-model such that $ t(Q)=\Sigma$. By Corollary \ref{aristogatti} we have that $g$ and $t$ are equivalent, and by Lemma \ref{belle} 
there exists   an  isomorphism of  pre-models  $\hat\eta\colon (Q,t,\vartheta) \to  (Z,g,\tau)$.
\end{proof}

\begin{definition}
We will call any injective pre-model $(Z,g,\tau)$ satisfying $g(Z)=\Sigma$ a {\sl  canonical  pre-model} associated with $\Sigma$.
\end{definition}

The backward step has a straightforward interpretation in terms of any canonical  pre-model associated with $\Sigma$, as the next result shows.
\begin{proposition}\label{few}
Let $\Sigma\subset \Lambda$ be an invariant canonical submanifold. Let $(Z,g,\tau)$ be a canonical  pre-model  associated with $\Sigma$. Let $x\in \Sigma$ and  let $z\in Z$ be such that $g(z)=x$. Then for all $m\geq 0$,  $$\sigma_m(x)=k_Z(z,\tau^{-m}(z))=k_Z(z,\tau^m(z)).$$
\end{proposition}
\begin{proof}
We have that $f^{-m}(x)\in \Sigma$. By Proposition \ref{mryellow},
$$k_X(f^{-n}(x),f^{-n-m}(x))=k_{f^n(X)}(x, f^{-m}(x))\to k_\Sigma(x, f^{-m}(x))\coloneqq k_Z(x, \tau^{-m}(z)).$$
Finally, $k_Z(z,\tau^{-m}(z))=k_Z(z,\tau^m(z))$ since $\tau$ is an automorphism.
\end{proof}
\begin{definition}
Let $Y$ be a  complex manifold and let $h\colon Y\to Y$ be a holomorphic self-map.  Let $y\in Y$. The {\sl divergence rate} is introduced in \cite[Definition 2.5]{AB} as $$c(h)\coloneqq \lim_{m\to\infty} \frac{k_Y(y, h^m(y))}{m}= \inf_{m\in\N} \frac{k_Y(y, h^m(y))}{m},$$  and  does not depend on $y\in Y$.
\end{definition}

Proposition \ref{few} immediately yields the following corollary.
\begin{corollary}\label{poi}
Let $\Sigma\subset \Lambda$ be an invariant canonical submanifold.  Let $(Z,g,\tau)$ be a canonical  pre-model associated with $\Sigma$.  Let $x\in \Sigma$.  Then
$$c(\tau)=\lim_{m\to \infty} \frac{\sigma_m(x)}{m}=\inf_{m\in \N} \frac{\sigma_m(x)}{m}.$$
\end{corollary}

The proof of Theorem \ref{mainintro} follows easily from the results of this section.

 We now want  to describe the dynamics on an invariant canonical submanifold.
We first need   to recall some definitions.

\begin{definition}\label{pinguino}
Let $Y$ a taut complex manifold.
We say that  the {\sl type} of a holomorphic self-map  $h\colon Y\to Y$ is
\begin{enumerate}
\item  {\sl elliptic} if the sequence $(h^n)$  is not compactly divergent (and hence $c(h)=0$),
\item  {\sl parabolic} if the sequence $(h^n)$  is compactly divergent and $c(h)=0$,
\item {\sl hyperbolic} if the sequence $(h^n)$  is compactly divergent and $c(h)>0$.
\end{enumerate}
The {\sl type}  of a  pre-model $(Q,t,\vartheta)$ is defined as the type of $\vartheta$. The {\sl type}  of an  invariant canonical submanifold $\Sigma$ is defined as the type of a canonical pre-model associated with $\Sigma$.
\end{definition}

We will need the following result proved in \cite[Theorem 1.1]{A2}.
\begin{theorem}\label{bambi}
Let $Y$ be taut, and let $h\colon Y\to Y$ be a holomorphic self-map. Then the following are equivalent:
\begin{enumerate}
\item the sequence $(h^n)$ is not compactly divergent,
\item the subset $\{h^n(y)\}$ is relatively compact in $Y$ for all $y\in Y$,
\item there exists $y\in Y$ such that the subset $\{h^n(y)\}$ is relatively compact in $Y$.
\end{enumerate}
\end{theorem}

The next results show that it is possible to detect the type of an invariant  canonical submanifold only by  looking at the backward steps at any  $x\in \Sigma$. This, together with Lemma \ref{slayer}, shows that every dynamical information concerning $\Sigma$ is encoded in the sequence $(\sigma_m(x))$.

\begin{proposition}\label{baghera}
Let  $\Sigma$ be an invariant canonical submanifold, and let $x\in \Sigma$. Then
\begin{enumerate}
\item the type of $\Sigma$ is elliptic if and only if the sequence $(\sigma_m(x))$ is bounded,
\item the type of $\Sigma$ is  parabolic if and only if the sequence $(\sigma_m(x))$ is unbounded and $\lim_{m\to \infty} \frac{\sigma_m(x)}{m}=0$,
\item the type of $\Sigma$ is  hyperbolic if and only if  $\lim_{m\to \infty} \frac{\sigma_m(x)}{m}>0$ and in this case
$$ \lim_{m\to \infty} \frac{\sigma_m(x)}{m}=c(\tau).$$
\end{enumerate}
\end{proposition}
\begin{proof}
Let     $(Z,g,\tau)$ be   a canonical pre-model associated with $\Sigma$, and let $z\in Z$ be such that $g(z)=x$.

(1) Assume that $\Sigma$ is elliptic.  Then by Theorem \ref{bambi} the subset $\{\tau^m(z)\}$ is contained in a compact subset $K\subset Z$. Hence the sequence $(k_Z(z,\tau^m(z)))=(\sigma_m(x))$ is bounded. Conversely, assume  that the sequence $(\sigma_m(x))=(k_Z(z,\tau^m(z)))$ is bounded by $M>0$. Then for all $m\geq 0$ we have $\tau^m(z)\subset \{w\in Z\colon k_Z(z,w)\leq M\}$ which is compact since $Z$ is complete  hyperbolic.

(2) Assume that $\Sigma$ is parabolic. Then  the sequence $(\tau^m(z))$ eventually leaves all compact subsets. Since $Z$ is complete  hyperbolic,   the sequence $(k_Z(z,\tau^m(z)))=(\sigma_m(x))$ is unbounded. By Corollary \ref{poi},  $\lim_{m\to \infty} \frac{\sigma_m(x)}{m}=0$. Conversely, if $(\sigma_m(x))$ is unbounded then $\Sigma$ cannot be elliptic, and sice $c(\tau)=0$ by Corollary \ref{poi}, we are done.

(3) It follows easily from Corollary \ref{poi}.
\end{proof}

\begin{definition}
Let $m\geq 1$. The forward $m$-step of $f$ at $x$ is defined as 
\begin{equation}\label{pignaforward}
s_m(x)\coloneqq \lim_{n\to\infty} k_X(f^n(x),f^{n+m}(x)).
\end{equation}
\end{definition}
Notice that the limit in (\ref{pignaforward}) exists since the sequence $$(k_X(f^{-n-m}(x),f^{-n}(x)))_{n\geq 0}$$ is monotonically decreasing.

\begin{proposition}\label{aladdin}
Let $\Sigma\subset \Lambda$ be an invariant canonical submanifold. Let $(Z,g,\tau)$ be a canonical pre-model  associated with $\Sigma$. Then $c(\tau)\geq c(f).$
\end{proposition}
\begin{proof}
By \cite[Proposition 2.7]{AB}, the divergence rate satisfies $c(f)=\lim_{n\to\infty}\frac{s_m(x)}{m}.$
The result follows  since $s_m(x)\leq \sigma_m(x)$ for all $m\geq 0$.
\end{proof}
The following corollary easily follows from Proposition \ref{aladdin}.
\begin{corollary}
Let $f\colon X\to X$ be a hyperbolic self-map. If $\Sigma$ is an invariant canonical submanifold, then it is hyperbolic. 
\end{corollary}
 For a proof of the following result, see, e.g.,  \cite[Theorem 2.1.29]{A}.
\begin{theorem}
Let $Y$ be a taut manifold, and let $h\colon Y\to Y$ be a holomorphic self-map such that the sequence $(h^n)$ is not compactly divergent. Then there exists a submanifold $M$ of $Y$ called the {\sl limit manifold} and a holomorphic retraction $\rho\colon Y\to M$ which is a  limit point of $(h^n)$ such that every holomorphic self-map $k\colon Y\to Y$ which is a limit point of $(h^n)$ is of the form $k=\gamma\circ \rho$, where $\gamma$ is an automorphism of $M$. Moreover $h(M)\subset M$ and $h|_M$ is an automorphism of $M$.
\end{theorem}

The next result characterizes the invariant canonical submanifolds of elliptic type.
\begin{proposition}
If $f\colon X\to X$ is elliptic then the limit manifold is an elliptic   invariant canonical submanifold.
 Conversely,  if an  invariant canonical submanifold $\Sigma\subset X$ is elliptic, then $f$ is elliptic and $\Sigma$ is the limit manifold. 

\end{proposition}
\begin{proof}
Let $f\colon X\to X$ be elliptic,  let $M\subset X$ be the limit manifold, and let $\iota\colon M\to X$ denote the inclusion mapping. Then $(M,\iota, f|_M )$ is an injective  pre-model for $f$. We show that it is elliptic. 
Let $z\in M$. Since $f$ is elliptic, by Theorem \ref{bambi} there exists a compact subset $K\subset X$ containing the subset $\{f^n(z)\}$. Since $M$ is a holomorphic retract of $X$, it is closed in $X$ and thus $K\cap M$ is compact. Thus by Theorem \ref{bambi} the  pre-model $(M,\iota, f|_M )$ is elliptic. 

We now show that the  pre-model $(M,\iota, f|_M )$ is canonical.
Assume by contradiction that $M$ is not a canonical submanifold. Then there exist a canonical submanifold $\Sigma\supsetneqq M$,   a canonical pre-model  $(Z,g,\tau)$ associated with $\Sigma$, and  a morphism
  $\hat\eta\colon (M,\iota, f|_M )\to (Z,g,\tau).$
The pre-model  $(Z,g,\tau)$ is elliptic: indeed, if $z\in M$ and if $K$ is a compact subset containing the subset $\{f^n(z)\}$, then
 $\tau^n(\eta (z))\in \eta (K)$ for all $n\geq 0$.
Let $x\in \Sigma$, $x\not\in M$, and let $w\in Z$ be such that $g(w)=x$. Since $\tau$ is an elliptic automorphism  there exists  a subsequence $(\tau^{n_k})$ converging uniformly on compact subsets to a holomorphic self-map $\mu\colon Z\to Z$. Since $Z$ is taut, by \cite[Proposition 2.1.24]{A} the self-map $\mu$ is an automorphism. Hence  $y\coloneqq \mu^{-1}(w)\in Z$ is such that $\tau^{n_k}(y)\to w$. This implies that 
$$f^{n_k}(g(y))=g(\tau^{n_k}(y))\to x.$$
Up to taking another subsequence we may assume that $(f^{n_k})$ converges uniformly on compact subsets to a mapping $h\colon X\to X$, such that $h(g(y))=x\not\in M$, which contradicts the assumption that $M$ is the limit manifold.

Conversely, let  $\Sigma\subset X$ be an elliptic invariant canonical submanifold. Let     $(Z,g,\tau)$ be  a canonical  pre-model associated with $\Sigma$, and let $z\in Z$. By Theorem \ref{bambi}, the subset  $\{\tau^m(z)\}$ is contained in a compact subset $K\subset Z$. Hence the subset $\{f^n(g(z))\}$ is contained in the compact subset $f(K)\subset X$, which implies that $(f^n)$ is not divergent on compact subsets. Thus $f$ is elliptic. 
Assume by contradiction that  the limit manifold  $M\subset X$ and $\Sigma$ do not coincide. By the previous part of the proof,  $M$ is also a canonical submanifold. Hence
$M$ and $\Sigma$ are disjoint, and arguing as before we contradict the assumption that $M$ is the limit manifold.
\end{proof}

\section{The unit ball}\label{mufasa}
In this section we prove Theorem  \ref{timon}. We first recall some definitions  and results about  the dynamics in $\B^q$.

\begin{definition}\label{jafar}

The Siegel upper half-space $\mathbb H^q$ is defined by $$\mathbb{H}^q=\left\{(z,w)\in \C\times \C^{q-1}, \Im(z)>\|w\|^2\right\}.$$ Recall that $\mathbb H^q$ is biholomorphic to the ball $\B^q$ via the {\sl Cayley transform} $\Psi\colon \B^q\to \H^q$ defined as $$\Psi(z,w)=\left(i\frac{1+z}{1-z},\frac{w}{1-z}\right), \quad (z,w)\in \C\times \C^{q-1}.$$

Let $\langle\cdot, \cdot\rangle$ denote the standard Hermitian product in $\C^q$. In several complex variables, the natural generalization of non-tangential limit at the boundary is the following.
 If $\zeta\in \partial\B^q$, then the set $$K(\zeta,R)\coloneqq\{z\in \B^q: |1-\langle z,\zeta\rangle|< R(1-\|z\|)\}$$ is a {\sl Kor\'anyi region} of {\sl vertex} $\zeta$ and {\sl amplitude} $R> 1$.
Let $f\colon \B^q \to \C^m$ be a holomorphic map. We say that $f$ has {\sl $K$-limit} $L\in \C^m$ at
$\zeta$ if for
each sequence $(z_k)\subset \B^q$ converging to $\zeta$ such that
$(z_k)$  belongs eventually to some Kor\'anyi region of vertex $\zeta$, we have
that $f(z_k)\to L$.

A sequence $(z_k)\subset \B^q$ converging to $\zeta\in \partial\B^q$ is said to be {\sl restricted} at $\zeta$ if   $\la z_k, \zeta\ra\to 1$ non-tangentially in $\D$, while  it is said to be {\sl special} at $\zeta$ if
\[
\lim_{k\to \infty}k_{\B^q}(z_k,\langle z_k,\zeta\rangle \zeta)=0.
\]

\end{definition}

\begin{definition}\label{magamago'}
A point $\zeta\in \partial \B^q$ such that $K\hbox{-}\lim_{z\to\zeta}f(z)=\zeta$ and
$$\liminf_{z\to\zeta}\frac{1-\|f(z)\|}{1-\|z\|}=\lambda<\infty$$ is called a {\sl boundary regular fixed point}, and $\lambda$ is called  its {\sl dilation}.
If  $\lambda>1$, then we call the point $\zeta$  a {\sl boundary repelling fixed point}.
\end{definition}

The following result from  \cite{H} generalizes the classical Denjoy--Wolff theorem in the unit disc.
\begin{theorem}
Let $f\colon \B^q\to \B^q$ be holomorphic. Assume that $f$ admits no fixed points in $\B^q$.
Then there exists a    boundary regular fixed point $p\in \de \B^q$ with dilation $\lambda\leq 1$, called the {\sl Denjoy--Wolff point} of $f$, such that $(f^n)$ converges uniformly on compact subsets to the constant map $z\mapsto p$. 
\end{theorem}

A proof of the following result is given in \cite[Theorem 2.4.20]{A} for the more general case of bounded convex domains.
\begin{theorem}
A holomorphic self-map $f\colon\B^q\to\B^q$ is elliptic if and only if it admits a fixed point $z\in \B^q$.
\end{theorem}

\begin{remark}\label{peterpan}
Let $f\colon\B^q\to\B^q$ be a holomorphic self-map  without fixed points, and let $\lambda$ be the dilation at its Denjoy--Wolff fixed point. Then by \cite[Proposition 5.8]{AB} the divergence rate of $f$ satisfyies $$c(f)=-\log \lambda.$$ Thus  Definition \ref{pinguino} generalizes the classical definition of elliptic,  parabolic and hyperbolic self-maps in the unit ball (see, e.g., \cite[Definition 5.3]{AB}).
\end{remark}

\begin{definition}
Let $f\colon \B^q\to \B^q$ be a holomorphic self-map. Let $\zeta\in \partial \B^q$ be a boundary regular fixed point. The {\sl stable subset} of $f$ at $\zeta$  is defined as 
 the subset consisting of all $z\in \B^q$ such that there exists a backward orbit with bounded step starting at $z$ and  converging to $\zeta$. We denote it by $\mathcal S(\zeta)$.
\end{definition}
Clearly  $\mathcal S(\zeta)$ coincides with the union of all backward orbits in $\B^q$ with bounded  step converging to $\zeta$.

We can now give an answer to Question \ref{robinhoodintro}.
\begin{proposition}
Let  $f\colon \B^q\to\B^q$ be univalent, and let $\zeta\in\de\B^q$ be a boundary regular fixed point. Then  $\mathcal{S}(\zeta)$, if non-empty, is the disjoint union of invariant canonical submanifolds, which are injectively immersed  holomorphic balls $\B^k$ with $1\leq k\leq q$.
\end{proposition}
\begin{proof}
Clearly $\mathcal{S}(\zeta)\subset \mathcal{S}(f)$. 
Assume that $\mathcal{S}(\zeta)$ intersects an invariant  canonical submanifold $\Sigma$. Thus there exists $z\in \Sigma$ such that $(f^{-n}(z))$ converges to $\zeta$. If $w\in \Sigma$, then the sequence $(k_{\B^q}(f^{-n}(z),f^{-n}(w)))$ is bounded, which means that $(f^{-n}(w))$ converges to the same point $\zeta$. Hence $\Sigma\subset \mathcal{S}(\zeta)$.
The result follows by Remark \ref{merlino}.
\end{proof}

In order to give a more precise answer to Question \ref{robinhoodintro} in the univalent case, one should answer the two following open questions.
\begin{question}
Let  $f\colon \B^q\to\B^q$ be univalent, and let 
 $\zeta\in\de\B^q$ be a boundary repelling fixed point with dilation $1<\lambda<\infty$. By \cite[Lemma 3.1]{O},  if $\zeta$ is isolated
  from other boundary repelling fixed points with dilation less or equal than $\lambda$, then $\mathcal{S}(\zeta)\neq \varnothing$. Is the same true if the point $\zeta$ is not isolated?

\end{question}
\begin{question}\label{ariel}
Let  $f\colon \B^q\to\B^q$ be univalent, and let 
 $\zeta\in\de\B^q$ be a boundary repelling fixed point.  Can  $\mathcal{S}(\zeta)$ contain two different invariant canonical submanifolds?
In other words, may there exist two backward orbits  $(f^{-n}(z)),(f^{-n}(w))$ with bounded step, converging to $\zeta$, and such that
$$k_{\B^q}(f^{-n}(z),f^{-n}(w))\to\infty?$$

\end{question}

\begin{remark}
If $\zeta\in\de\B^q$ is a boundary regular fixed point with dilation $\lambda\leq 1$, then $\mathcal{S}(\zeta)$ may  contain two different invariant canonical submanifolds. For example, consider the domain of $\C$ defined by $A\coloneqq \C\smallsetminus {\R^-}$. Consider the holomorphic univalent self-map $f\colon A\to A$ defined by $f(z)=z+1$. Then we have $\Lambda= \C\smallsetminus \R$, and the upper and lower half-planes are canonical invariant submanifolds contained in the stable subset of the Denjoy--Wolff point of $f$.
\end{remark}

Let  $f\colon \B^q\to\B^q$ be univalent, and let $\zeta\in\de\B^q$ be a boundary regular fixed point. 
Let  $\Sigma$ be an   invariant canonical submanifold in the stable subset $\mathcal{S}(\zeta)$. What can be said about the type of  $\Sigma$? If the point $\zeta$ is repelling, we have the following answer.

\begin{proposition}\label{biancaebernie}
Let  $f\colon \B^q\to\B^q$ be univalent, and let $\zeta\in\de\B^q$ be a boundary repelling fixed point, with dilation $1<\lambda<\infty$.
Let  $\Sigma$ be a  canonical submanifold in the stable subset $\mathcal{S}(\zeta)$,  let $1\leq k\leq q$ be the dimension of $\Sigma$,
 let $x\in\Sigma$ and  $$\mu\coloneqq \lim_{m\to \infty} e^{\frac{\sigma_m(x)}{m}}\geq \lambda.$$
Then $\Sigma$ is hyperbolic, and
 there exists  a $(k-1)\times(k-1)$ diagonal unitary matrix $U$,  and 
  an injective holomorphic immersion  $g\colon \H^k\to \B^q$ with $g(\H^k)=\Sigma$ and $$K\hbox{-}\lim_{z\to \infty}g(z)=\zeta,$$
 such that
 $$\left(\H^k,g,\tau\colon(z,w)\mapsto \left(\frac{1}{\mu} z,\frac{1}{\sqrt \mu}\, Uw\right)\right)$$ is an injective pre-model for $f$.

\end{proposition}
\begin{proof}
Let $n\geq 0$.
Since $\lambda^n$ is the dilation at $\zeta$ of the mapping $f^n$, we have, for any $w\in \B^q$ (see, e.g., \cite{A}),
$$n\log \lambda=\liminf_{z\to \zeta}(k_{\B^q}(w,z)-k_{\B^q}(w,f(z))).$$
Since $$k_{\B^q}(w,z)-k_{\B^q}(w,f(z))\leq k_{\B^q}(z,f(z)),$$ we have that 
 $n\log\lambda\leq {\sigma_n(x)}$, that is,
$\lambda\leq e^{\frac{\sigma_n(x)}{n}}$. Thus $\mu\geq\lambda$.
Let $(Q, t,\vartheta)$ be a canonical pre-model associated with $\Sigma$. By Remark \ref{merlino}, $Q$ is biholomorphic to $\B^k$, where $1\leq k\leq q$.   By Corollary \ref{poi}, the divergence rate of the automorphism $\vartheta$  satisfies $c(\vartheta)=\log\mu$. Hence $\vartheta$ is hyperbolic and  by Remark \ref{peterpan} the dilation at its Denjoy--Wolff point is equal to $e^{-c(\vartheta)}=\frac{1}{\mu}$, and thus there exists (see e.g. \cite{A}) a biholomorphism $h\colon Q\to \H^k$ such that $$\tau\coloneqq h\circ  \vartheta\circ h^{-1} (z,w)= \left(\frac{1}{\mu} z,\frac{1}{\sqrt \mu} \,Uw\right),$$ where $U$ is a $(k-1)\times(k-1)$ diagonal unitary matrix. Set $g\coloneqq t\circ h^{-1}$. Then $(\H^k,g,\tau)$ is an injective  pre-model isomorphic to $(Q, t,\vartheta)$.

We now address the regularity at $\infty$ of the intertwining mapping $g$. Let $(z_n,w_n)$ be a backward orbit in $\H^k$ for $\tau$. Then  $(z_n,w_n)$ converges to $\infty$ and there exists $C>0$ such that  
$$k_{\H^k}((z_n,w_n),( z_{n+1},w_{n+1}))\leq C,\quad \mbox{and}\quad k_{\H^k}((z_n,w_n),(z_n, 0))\leq C.$$ 
 Clearly $g(z_n,w_n)$ is a backward orbit for $f$ which converges to $\zeta\in\partial\B^q$. Then  \cite[Theorem 5.6]{AB}  yields the result.
\end{proof}

It is natural to ask whether, using the notations of the previous proposition, some condition implying  $\mu=\lambda$ can be given.
For example, if there exists $x\in \Sigma$ such that $\sigma_1(x)=\log\lambda$, then $\mu=\lambda$. This follows immediately from $\lambda\leq \mu= \inf_{m\in\N} e^{\frac{\sigma_m(x)}{m}}$. Notice that in this case, by \cite[Lemma 3.7]{O}, the backward orbit $(f^{-n}(x))$ is special and $\langle f^{-n}(x),\zeta \rangle\to 1$ asymptotically radially in $\D$, thus in particular  $(f^{-n}(x))$ is restricted. 
The next result shows that the special and restricted convergence of  $(f^{-n}(x))$ is in fact enough to obtain $\lambda=\mu$. 

\begin{proposition}
Let $f\colon \B^q\to\B^q$ be a univalent self-map. 
Let $\zeta\in\de\B^q$ be a boundary repelling fixed point, with dilation $1<\lambda<\infty$. Let $\Sigma$ be an invariant canonical submanifold contained in $\mathcal{S}(\zeta)$, and assume there exists $x\in \Sigma$ such that the backward orbit $(f^{-n}(x))$ is special and restricted. Then  $\mu=\lambda$.
\end{proposition}

\begin{proof}
Let $(z_{n_k})$ be a subsequence such that the limit $$\lim_{k\to \infty}\frac{1-\langle z_{n_k}, e_1\rangle }{|1-\langle z_{n_k}, e_1\rangle|}=e^{i\vartheta},$$ with $\vartheta\in (-\pi/2,\pi/2)$. Then from   \cite[Proposition 5.4]{AB} (whose proof can be applied to the case of boundary regular fixed points) it follows that 
$$\sigma_1(x)=\lim_{k\to\infty}k_{\B^q}(z_{n_k}, f(z_{n_k}))=\log\frac{|e^{-2i\vartheta}+\lambda|+|1-\lambda|}{|e^{-2i\vartheta}+\lambda|-|1-\lambda|}.$$
Similarly, for all $m\geq 0$,
$$\sigma_m(x)=\lim_{k\to\infty}k_{\B^q}(z_{n_k}, f^m(z_{n_k}))=\log\frac{|e^{-2i\vartheta}+\lambda^m|+|1-\lambda^m|}{|e^{-2i\vartheta}+\lambda^m|-|1-\lambda^m|}.$$
Thus $$\lim_{m\to\infty} \frac{\sigma_m(x)}{m}=\lim_{m\to\infty}\log\frac{\sqrt[m]{{|e^{-2i\vartheta}+\lambda^m|+|1-\lambda^m|}}}{\sqrt[m]{|{e^{-2i\vartheta}+\lambda^m|-|1-\lambda^m|}}}.$$
We have $$\sqrt[m]{{|e^{-2i\vartheta}+\lambda^m|+|1-\lambda^m|}}\to\lambda,$$ and since 
$$\Re (e^{-2i\vartheta}+1)|\leq |e^{-2i\vartheta}+\lambda^m|-|1-\lambda^m|\leq |e^{-2i\vartheta}+1|,$$
the result follows.

\end{proof}

The proof of Theorem \ref{timon} follows easily from the results of this section.

The following univalent self-map of the Siegel upper half-space was first studied in \cite[Example 6.3]{O} as an example of a self-map admitting a real one-dimensional curve of boundary repelling fixed points with the same dilation. We study its stable subset.
\begin{example}
Let $f\colon \H^2\to \H^2$ be defined by $f(z,w)=(2z+iw^2,w).$ Then $f$ is univalent and hyperbolic, with Denjoy--Wolff point at infinity. We have
$$f^n(z,w)=(2^nz+(2^n-1)iw,w).$$
Hence $(z,w)\in \H^2$ is in the invariant subset $\Lambda$ if and only if $-\Im iw^2\geq|w|^2$, that is, if and only if $w$ is  pure imaginary.
Then $\Lambda=\H^2\cap\{\Re w=0\}$, and it is therefore a real three-dimensional submanifold of $\H^2$. The stable manifold $\mathcal{S}(f)$ coincides with $\Lambda$ since every backward orbit has bounded step.
Let $r\in \R$  and define $\Sigma_{r} \coloneqq\H^2\cap\{w=ir\} $. Then $\Sigma_r$ is a canonical invariant submanifold.
The point $(ir^2,ir)\in \partial \H^2$ is a boundary repelling fixed point with dilation $2$ and  $\Sigma_r$ is the stable subset of $f$ at $(ir^2,ir)$. 

We  want now to describe a canonical pre-model associated with $\Sigma_r$.
The automorphism of $\H^2$ given by $$h_r(z,w)=(z+ir^2-2rw,w-ir)$$ maps the boundary repelling point $(ir^2,ir)$ to the origin and $\Sigma_r$  to $\Sigma_0$.
On the other hand, $$h_r\circ f\circ h_r^{-1}=f.$$ A canonical pre-model  associated with $\Sigma_0$ is given by $(\H, g,\tau\colon z\mapsto 2z)$,  where   $g(z)=(z,0)$. Hence a canonical pre-model  associated with $\Sigma_r$ is given by
$$(\H, h_r^{-1}\circ g,\tau\colon z\mapsto 2z).$$

\end{example}
\bibliographystyle{amsplain}

\end{document}